\documentclass[twoside]{article}
\usepackage[utf8]{inputenc}
\usepackage{tocloft}

\usepackage[utf8]{inputenc}
\usepackage[english]{babel}
\usepackage[pdfborder = {0 0 0}]{hyperref}
\usepackage[T1]{fontenc}
\usepackage{a4wide}
\usepackage{setspace}
\usepackage{amsfonts,amsmath,amssymb,amsthm}
\usepackage{graphicx}
\usepackage{graphics}
\usepackage[center]{caption}
\usepackage{mathrsfs}
\usepackage{bbm}
\usepackage{amsmath}
\usepackage{hyperref}
\usepackage[export]{adjustbox}

\usepackage{fancyhdr} 	
\DeclareMathOperator*{\essinf}{ess\,inf}

\graphicspath{{./}} 

\newtheorem{defi1}{Definition}[section]

\def\R{\mathbb{R}}

\DeclareMathOperator{\dist}{dist}

\newcommand{\eps}{\varepsilon}

\newtheorem{satz}[defi1]{Theorem}

\newtheorem{prop}[defi1]{Proposition}

\newtheorem{lemma}[defi1]{Lemma}

\newtheorem{cor}[defi1]{Corollary}
\newtheorem{rem}[defi1]{Remark}

\title{\large\textbf{On a fractional boundary version of Talenti's inequality in the unit ball}}
\author{\normalsize YASSIN EL KARROUCHI AND TOBIAS WETH}
\date{}
\begin{document}
\maketitle
\begin{abstract}
Inspired by recent work of Ferone and Volzone \cite{ferone2021symmetrization}, we derive sufficient conditions for the validity and non-validity of a boundary version of Talenti's comparison principle in the context of Dirichlet-Poisson problems for the fractional Laplacian $(-\Delta)^s$ in the unit ball $\Omega= B_1(0) \subset \R^N$. In particular, our results imply a universial failure of the classical pointwise Talenti inequality in the fractional radial context which sheds new light on the one-dimensional counterexamples given in \cite{ferone2021symmetrization}. 
  \end{abstract}
\section{Introduction}

Let $\Omega \subset \R^N$ be a smooth bounded domain, and consider the classical Dirichlet-Poisson problem 
\begin{equation}
\label{classical-Dirichlet-Poisson}  
\left\{\begin{aligned}
\ -\Delta u &= f && \qquad  \text{in $\Omega$}, \\ 
u&= 0 &&\qquad \text{on $\partial \Omega$.}  
\end{aligned}
\right.
\end{equation}
It is well known that for every $f \in L^\infty(\Omega)$, the problem (\ref{classical-Dirichlet-Poisson}) has a unique (weak) solution $u_{f} \in H^1_0(\Omega)$, and in fact we have $u_{f} \in C^{1,\alpha}(\overline \Omega)$ for some $\alpha>0$ by classical elliptic regularity theory. In the case where $f$ is nonnegative, one may compare (\ref{classical-Dirichlet-Poisson}) with the corresponding Schwarz symmetrized problem  
\begin{equation}
\label{classical-Dirichlet-Poisson-symmetrized}  
\left\{\begin{aligned}
\ -\Delta u &= f^* && \qquad  \text{in $\Omega^*$}, \\ 
u&= 0 &&\qquad \text{on $\partial \Omega^*$.}  
\end{aligned}
\right.
\end{equation}
Here $f^*$ denotes the Schwarz symmetrization of $f$, which is defined on the Schwarz-symmetrized domain $\Omega^*= B_r(0)$, where $r>0$ is chosen such that $|\Omega|= |B_r(0)|$. A special case of the classical {\em Talenti comparison principle}  then relates the unique solutions $u_{f}$ and $u_{f^*}$ of (\ref{classical-Dirichlet-Poisson}) and (\ref{classical-Dirichlet-Poisson-symmetrized}). More precisely, the main theorem in the seminal paper \cite{ASNSP_1976_4_3_4_697_0} yields the pointwise inequality
\begin{equation}
  \label{eq:talenti-claim-classical}
\bigl(u_{f}\bigr)^* \le u_{f^*}   \qquad \text{in $\Omega^*$.}
\end{equation}
In fact, as shown in \cite[Theorem 1]{ASNSP_1976_4_3_4_697_0}, this inequality holds for a rather general class of elliptic second order operatos in place of the Laplacian. It has many applications to optimization problems in PDE and the calculus of variations and has been extended to the context of more general elliptic and parabolic boundary value problems, see \cite{alvino-trombetti-lions-1990}.

Clearly, the inequality (\ref{eq:talenti-claim-classical}) has the following immediate consequence regarding the outer normal derivatives of the unique solutions $u_{f}$ and $u_{f^*}$ of (\ref{classical-Dirichlet-Poisson}) and (\ref{classical-Dirichlet-Poisson-symmetrized}): 
\begin{equation}
\label{talenti-classical-immediate-consequence}  
0 \le -\partial_\nu \bigl(u_{f}\bigr)^* \le -\partial_\nu u_{f^*} \qquad \text{on $\partial \Omega^*$}
\end{equation}
The main purpose of the present paper is to study the validity and failure of this inequality -- depending on $f$ -- in the context of fractional Dirichlet-Poisson problems in a ball $\Omega= \Omega^*$, where the (negative) normal outer derivatives are replaced by fractional normal derivatives.
Our work is strongly motivated by recent, highly interesting work \cite{ferone2021symmetrization} of Volzone and Ferone on a nonlocal version of Talenti's theorem. In particular, Volzone and Ferone \cite{ferone2021symmetrization} considered, for $s \in (0,1)$,  the fractional Dirichlet-Poisson problem
\begin{equation}
\label{fractional-Dirichlet-Poisson}  
\left\{
\begin{aligned}
(-\Delta)^su &= f &&\qquad \text{in $\Omega,$} \\
u& = 0 &&\qquad \text{on $\mathbb{R}^N\setminus\Omega$}
\end{aligned}
\right.
\end{equation}
and its symmetrized version
\begin{equation}
\label{fractional-Dirichlet-Poisson-symmetrized}  
\left\{
\begin{aligned}
(-\Delta)^su &= f^* &&\qquad \text{in $\Omega^*,$} \\
u& = 0 &&\qquad \text{on $\mathbb{R}^N\setminus\Omega^*$.}
\end{aligned}
\right.
\end{equation}
Here, again, we let $f \in L^\infty(\Omega)$ be a nonnegative function. Moreover, $(-\Delta)^su$ denotes the usual fractional Laplacian, which for sufficiently smooth and integrable functions $u$ is defined pointwisely as 
\begin{equation}
  \label{eq:pointwise-def-fractional-laplace}
(-\Delta)^su(x)=c_{N,s} \lim_{\eps \to 0^+} \int_{\mathbb{R}^N \setminus B_{\eps}(x)}\frac{u(x)-u(y)}{|x-y|^{N+2s}}dy, \qquad c_{N,s}=\frac{s4^s\Gamma(\frac{N}{2}+s)}{\pi^{\frac{N}{2}}\Gamma(1-s)}.
\end{equation}
Note also that $(-\Delta)^su \in L^2(\R^N)$ is well-defined for functions $u \in H^{s}(\mathbb{R}^N)$ via Fourier transform, i.e. 
\[
\widehat{(-\Delta)^su}(\xi)=|\xi|^{2s}\widehat{u}(\xi) \qquad \text{for $\xi \in \R^N$.}
\]
The natural energy space for weak solutions of (\ref{fractional-Dirichlet-Poisson}) is the space $\mathscr{H}^s_0(\Omega)$, defined as the set of functions $u \in H^s(\R^N)$ with $u \equiv 0$ on $\R^N \setminus \Omega$. More precisely, $u \in \mathscr{H}^s_0(\Omega)$ is, by definition, a weak solution of (\ref{fractional-Dirichlet-Poisson}) if
\begin{equation}
  \label{eq:defining-weak-sol}
\int_{\R^N}|\xi|^{2s}\widehat{u}(\xi)\widehat{w}(\xi)\,d\xi = \int_{\Omega}f w\,dx \qquad \text{for all $w \in \mathscr{H}^s_0(\Omega)$.}
\end{equation}
We note that (\ref{classical-Dirichlet-Poisson}) can be regarded as a limit problem for (\ref{fractional-Dirichlet-Poisson}) in the limit $s \to 1^-$. It is therefore appriopriate to let, whenever $s \in (0,1)$ is fixed and $f \in L^\infty(\Omega)$ is given, $u_{f} \in \mathscr{H}^s_0(\Omega)$ also denote the unique weak solution of (\ref{fractional-Dirichlet-Poisson}).
By fractional elliptic regularity theory up to the boundary (see \cite{ros2014dirichlet}), we then have $u_{f} \in C^s(\R^N)$. Moreover, setting $\delta(x):= \dist(x,\partial \Omega)$ for $x \in \Omega$, the function $x \mapsto \frac{u_{f}(x)}{\delta^s(x)}$ extends to a function in $C^\alpha(\overline \Omega)$ for some $\alpha>0$, and its restriction
  $$
  \frac{u_{f}}{\delta^s}: \partial \Omega \to \R
  $$
  to the boundary $\partial \Omega$ is usually called the {\em fractional normal derivative} of $u_{f}$. Note here that the limiting case $s=1$ leads again to the normal derivative. More precisely, for a function $u \in C^1(\overline \Omega)$ with $u \equiv 0$ on $\partial \Omega$ we have $\frac{u}{\delta} = -\partial_\nu u$ on $\partial \Omega$, where $\nu$ denotes the outer unit normal on $\partial \Omega$.

In \cite{ferone2021symmetrization}, Volzone and Ferone relate the unique solutions $u_{f}$ and $u_{f^*}$ of (\ref{fractional-Dirichlet-Poisson}) and (\ref{fractional-Dirichlet-Poisson-symmetrized}) by proving the mass concentration inequality 
\[
\int_{B_r(0)} (u_f)^* dx \leq \int_{B_r(0)} u_{f^*} dx \ \ \ \ \ \ \text{for every} \ \ r>0
\]
Note that this is a weaker conclusion than the pointwise inequality (\ref{eq:talenti-claim-classical}). In fact, Volzone and Ferone \cite{ferone2021symmetrization} also provide one-dimensional counterexamples which show that (\ref{eq:talenti-claim-classical}) fails in general. More precisely, in the one-dimensional case where $\Omega = \Omega^* = (-1,1) \subset \R$, the functions $u_f$ and $u_{f^*}$ were computed for two explicit even functions $f \in L^\infty(\Omega)$, and the failure of (\ref{eq:talenti-claim-classical}) was shown in this explicit examples. Our first main result of this paper shows that this failure is in fact -- in any space dimension -- a {\em universal effect} in the radial fractional case. Before stating this result, we wish to add an observation regarding (\ref{talenti-classical-immediate-consequence}) in the case $s=1$ if $\Omega= \Omega^*$ is a ball and $f$ is a radial nonnegative function. Actually, in this case we have the equality 
\begin{equation}
  \label{eq:interesting-observation-boundary}
\partial_\nu u_{f}  = \partial_\nu \bigl(u_{f}\bigr)^* = \partial_\nu u_{f^*}  \qquad \text{on $\partial \Omega$.}
\end{equation}
Indeed, the first equality follows from the fact that $u_f$ equals $\bigl(u_{f}\bigr)^*$ in a neighborhood of $\partial \Omega$, while the second equality follows from (\ref{talenti-classical-immediate-consequence}) and the integral identity
$$
\int_{\partial \Omega}\partial_\nu u_{f} d\sigma = \int_{\partial \Omega}\partial_\nu u_{f^*}d\sigma.  
$$
which in turn follows from (\ref{classical-Dirichlet-Poisson}), (\ref{classical-Dirichlet-Poisson-symmetrized}), the divergence theorem and the fact that $\int_{\Omega} f\,dx = \int_{\Omega}f^*\,dx$.

From now on, we fix the domain $\Omega = \Omega^*= B_1(0)$. 
Our first main result reads as follows.

\begin{satz}
  \label{first-main-theorem}
  Let $s \in (0,1)$ and $\Omega = B_1(0) \subset \R^N$.
  \begin{itemize}
  \item[(i)]   For every radial and nonnegative function $f \in L^\infty(\Omega)$, we have 
  \begin{equation}
    \label{eq:fractional-boundary-pointwise-inequality}
\frac{(u_f)^*}{\delta^s} \ge \frac{u_{f^*}}{\delta^s} \qquad \text{on $\partial \Omega$,}
  \end{equation}
  and equality holds if and only if $f = f^*$ a.e. in $\Omega$.
  \item[(ii)] For every radial nonnegative function $f \in L^\infty(\Omega)$ with $f \not  \equiv f^*$, there exist points $r \in (0,1)$ with the property that 
\begin{equation}
  \label{eq:talenti-pointwise-fractional-simple-opposite}
u_{f^*}< (u_f)^* \qquad \text{in $\Omega \setminus B_{r}(0)$,}
\end{equation}
so (\ref{eq:talenti-claim-classical}) fails. 
  \end{itemize}
\end{satz}
Theorem~\ref{first-main-theorem} shows that, in the fractional radial setting, the pointwise Talenti inequality (\ref{eq:talenti-claim-classical}) holds in all of $\Omega= \Omega^*$ {\em if and only if} $f$  is already Schwarz symmetric, i.e., if $f$ is non-increasing in the radial variable. We point out that the characterization of the equality case in (\ref{eq:fractional-boundary-pointwise-inequality}) is in striking contrast to (\ref{eq:interesting-observation-boundary}) and shows that the fractional case $s \in (0,1)$ differs fundamentally from the local case $s=1$.

Theorem~\ref{first-main-theorem} leaves open the question whether the reverse version of (\ref{eq:fractional-boundary-pointwise-inequality}), namely the inequality 
  \begin{equation}
    \label{eq:fractional-boundary-pointwise-talenti-inequality}
\frac{(u_f)^*}{\delta^s} \le \frac{u_{f^*}}{\delta^s} \qquad \text{on $\partial \Omega$,}
  \end{equation}
  might hold for certain nonradial functions $f$ defined on $\Omega = \Omega^* = B_1(0)$. Note that 
(\ref{eq:fractional-boundary-pointwise-talenti-inequality}) should be seen as a boundary version of (\ref{eq:talenti-claim-classical}), so we will call (\ref{eq:fractional-boundary-pointwise-talenti-inequality}) a {\em boundary Talenti inequality} and (\ref{eq:fractional-boundary-pointwise-inequality}) a {\em reverse boundary Talenti inequality}. We have the following result regarding a class of nonradial functions.
  
\begin{satz}
  \label{second-main-theorem}
Let $\Omega = B_1(0) \subset \R^N$, let $\xi \in \Omega \setminus \{0\}$, and let $0 < \rho <\min \{|\xi|, 1-|\xi|\} $ satisfy
\begin{equation}
  \label{eq:rho-condition-intro}
\bigl(1-(|\xi|-\rho)^2\bigr)^s \le \Bigl(1 - \frac{\rho}{1-|\xi|}\Bigr)^N  
\end{equation}
Then for every nonnegative function $f \in L^\infty(\Omega)$ supported in $B_{\rho}(\xi)$, the boundary Talenti inequality (\ref{eq:fractional-boundary-pointwise-talenti-inequality}) holds, and the inequality is strict if $f \not \equiv 0$. 
\end{satz}

For every $\xi \in \Omega \setminus \{0\}$, the assumption (\ref{eq:rho-condition-intro}) is clearly satisfied for $\rho$ sufficiently small (depending on $|\xi|$ and $1-|\xi|$). We note that, by approximation, Theorem~\ref{second-main-theorem} implies a symmetrization inequality for the Green function $G$ for (\ref{fractional-Dirichlet-Poisson}) on $\Omega = B_1(0)$. Recall that, for fixed $\xi \in \Omega \setminus \{0\}$, the function $G(\cdot,\xi)$ is a distributional solution of the problem
\begin{equation}
\label{distributional-Dirichlet-Poisson}  
\left\{
\begin{aligned}
(-\Delta)^su &= \delta_\xi &&\qquad \text{in $\Omega,$} \\
u& = 0 &&\qquad \text{on $\mathbb{R}^N\setminus\Omega$}
\end{aligned}
\right.
\end{equation}
where $\delta_\xi$ denotes the Dirac measure at the point $\xi$. Note that, at least formally, we have $(\delta_{\xi})^* = \delta_0$, and therefore the strict boundary Talenti inequality related to this distributional problem reads as follows:
\begin{equation}
  \label{eq:green-boundary-talenti}
\frac{G(\cdot,\xi)^*}{\delta^s}< \frac{G(\cdot,0)}{\delta^s} \quad \text{on $\partial \Omega$}\qquad \text{for every $\xi \in \Omega \setminus \{0\}$.}  
\end{equation}
We will establish this inequality as a byproduct of the estimates needed in the proof of Theorem~\ref{second-main-theorem}, see Proposition~\ref{sym-explizit-boundary-expr-green} below.

The contrast between (\ref{eq:interesting-observation-boundary}) and (\ref{eq:fractional-boundary-pointwise-inequality}) raises the question if $s$ is a borderline order for (\ref{eq:fractional-boundary-pointwise-inequality}) and if this inequality might be reversed in the case of the associated higher order problem $s>1$. Indeed this is the case, as we show in Theorem~\ref{first-main-theorem-s-greater-1}, and we note that the inequality given there also holds for the polyharmonic Dirichlet problem which corresponds to the case where $s$ is an integer greater than or equal to $2$.

This work is organized as follows. Section~\ref{sec:prel-notat} contains preliminaries and basic notations. In Section \ref{sec:reverse-bound-talent} we will complete the proof of Theorem~\ref{first-main-theorem}, and Section~\ref{sec:bound-talenti-ineq} is devoted to the proof of Theorem~\ref{second-main-theorem}.

\section{Preliminaries and notation}
\label{sec:prel-notat}
In this section, we introduce some basic notation, and we collect some preliminary results on the fractional maximum principle, Green and Martin kernels, and on Schwarz symmetrization. We start with the following lemma, proved in \cite[Lemma 7.3]{rosoton2015nonlocal}.

\begin{lemma}(Fractional strong maximum principle and Hopf's lemma)\label{Hopf-lemma}
Let $\Omega\subset\mathbb{R}^N$ be a bounded $C^{1,1}$ domain, $f \in L^\infty(\Omega)$ nonnegative, and let $u= u_f$ be the unique weak solution of (\ref{fractional-Dirichlet-Poisson}). Then, either
\[
\inf_{\overline \Omega} \frac{u}{\delta^s} >0 \qquad \text{or}\qquad u\equiv 0 \quad \text{in $\Omega$.}
\]
\end{lemma}
Here, as before $\delta(x)= \dist(x,\partial \Omega)$ for $x \in \Omega$, and we consider the unique extension of the function
$x \mapsto \frac{u(x)}{\delta^s(x)}$ to $\overline \Omega$.

Next we restrict our attention to the case where $\Omega= B_1(0)$. In this case, for given $f \in L^\infty(\Omega)$, the unique weak solution $u = u_f$ of (\ref{fractional-Dirichlet-Poisson}) is given pointwisely by
\begin{equation}
  \label{eq:green-function-formula}
u(x)= \int_{\Omega} G_s(x,y)f(y)\,dy,
\end{equation}
where $G_s$ is the fractional Green function for $(-\Delta)^s$ in $\Omega= B_1(0)$, for $s \in (0,1)$. Clearly, (\ref{eq:green-function-formula}) also holds in the case $s=1$, i.e., for the unique weak solution of (\ref{classical-Dirichlet-Poisson}) if $G_1$ is the classical Green function for $-\Delta$ in $\Omega$. We recall the following explicit form of $G_s$ which has been given in \cite{2010_Blumenthal}.
\begin{lemma}
  \label{sec:prel-notat-1}
Let $s \in (0,1]$. If $N\neq 2s$, we have 
\begin{align}\label{green-func}
    G(x,y)=\kappa_{N,s}|x-y|^{2s-N}\int_0^{r_0(x,y)}\frac{t^{s-1}}{(t+1)^{\frac{N}{2}}}dt
\end{align}
with
\begin{equation}
  \label{def-kappa-N-s}
\kappa_{N,s}:=\frac{\Gamma(\frac{N}{2})}{\pi^{\frac{N}{2}}4^s\Gamma(s)^2}\qquad \text{and}\qquad r_0(x,y)=\frac{(1-|x|^2)(1-|y|^2)}{|x-y|^2}, 
\end{equation}
while for $N=2s$ we have 
\begin{align}\label{green-func-1/2}
    G(x,y)=\kappa_{N,s}\log\Big(\frac{1^2-xy+\sqrt{(1-x^2)(1-y^2)}}{|x-y|}\Big).
\end{align}
\end{lemma}

As a consequence of (\ref{eq:green-function-formula}), the function $u = u_f$ satisfies, for given $f \in L^\infty(\Omega)$,
$$
\frac{u(x)}{\delta^s(x)} = \int_{\Omega}\frac{G(x,y)}{\delta^s(x)}f(y)\,dy\qquad  \text{for $x \in \Omega$.}
$$
By Lebesgue's theorem and straightforward estimates, a boundary version of this formula can be derived. More precisely, we have
\begin{equation}
  \label{eq:martin-kernel-boundary-formula}
\frac{u(\vartheta)}{\delta^s(\vartheta)} = \int_{\Omega}M_s(y,\vartheta)f(y)\,dy \qquad \text{for $\vartheta \in \partial \Omega$.}
\end{equation}
with the {\em fractional Martin kernel}
\begin{align*}
M_s: \Omega \times \partial \Omega \to \R, \qquad M_s(y,\vartheta):=\lim_{z\rightarrow\vartheta, z\in B_1}\frac{G_s(y,z)}{\delta^s(z)}=\lim_{z\rightarrow\vartheta, z\in B_1}\frac{2G_s(y,z)}{(1-|z|^2)^s}  \quad \text{for $y\in \Omega, \vartheta\in \partial \Omega$.}
\end{align*}
An explicit formula for the Martin kernel can be derived from Lemma~\ref{sec:prel-notat-1} for $\Omega= B_1(0)$. More precisely, it has been proved in \cite{abatangelo2018green} that for $s > 0$ and $N\geq 1$ we have 
\begin{equation}
  \label{explicit-martin-ball}
        M_s(x,\vartheta)=\frac{2\kappa_{N,s}}{s}\frac{(1-|x|^2)^s}{|\vartheta-x|^N}\ \ \ \ \ \ \ \ \text{for} \ \ x\in B_1(0), \vartheta\in \partial B_1(0),
    \end{equation}
    where $\kappa_{N,s}$ is defined in (\ref{def-kappa-N-s})\footnote{Note that the definition of the Martin kernel in \cite{abatangelo2018green} differs by a factor $2$ from our definition}. In the case $s=1$, this formula reduced to the representation of the classical Poisson kernel of the Dirichlet-Laplacian on $\Omega= B_1(0)$, which is given by
\begin{equation}
  \label{explicit-poisson-ball}
    P(x,\vartheta)= \frac{1}{\omega_{N-1}}\frac{1-|x|^2}{|\vartheta-x|^N} \qquad \text{for} \ \ x\in B_1(0), \vartheta\in \partial B_1(0),
\end{equation}
Here and in the following, $\omega_{N-1}= \frac{2 \pi^{N/2}}{\Gamma(N/2)}$ denotes the $N-1$-dimensional measure of the unit sphere $S^{N-1}$. Since the classical Poisson kernel satisfies the well-known identity
$$
\int_{\partial \Omega} P(x,\vartheta)\,d\sigma(\theta)= 1 \qquad \text{for all $x \in \Omega$,}
$$
it follows that
\begin{equation*}
\int_{\partial B_1(0)} \frac{1}{|\vartheta-x|^N}d\sigma(\vartheta) = \frac{\omega_{N-1}}{1-|x|^2} \qquad \text{for all $x \in B_1(0)$}
\end{equation*}
and therefore
\begin{equation}
  \label{eq:formula-spherical-integral-martin-kernel}
\int_{\partial B_1(0)}M_s(x,\vartheta)d\sigma(\vartheta) = \frac{2\kappa_{N,s}\, \omega_{N-1} }{s}(1-|x|^2)^{s-1} \qquad \text{for all $x \in B_1(0)$.}
\end{equation}

This identity gives rise to the following formula for the fractional boundary derivative of a solution of (\ref{fractional-Dirichlet-Poisson})  in the case where $\Omega= B_1(0)$ and $f$ is a radial function on $\Omega$.

\begin{lemma}
  \label{sec:further-results-1-0}
Let $\Omega= B_1(0)$, and let $f \in L^\infty(\Omega)$ be a radial function. Then  the unique solution $u = u_f$ of (\ref{fractional-Dirichlet-Poisson}) satisfies
  \begin{equation}
    \label{eq:general-radial-fractional-boundary-deriv}
  \frac{u_f}{\delta^s} \equiv \frac{2 \kappa_{N,s}}{s} \int_{\Omega} f(y)(1-|y|^2)^{s-1}dy \qquad \text{on $\partial \Omega$.}  
\end{equation}
\end{lemma}  

\begin{proof}
  Since $u:=u_f$ is a radial function, we can use (\ref{eq:martin-kernel-boundary-formula}) and (\ref{eq:formula-spherical-integral-martin-kernel}) to write the unique value of $\frac{u_f}{\delta^s}$ on $\partial \Omega$ as 
  \begin{align*}
    \frac{u_f}{\delta^s} &=\frac{1}{\omega_{N-1}} \int_{S^{N-1}} \int_{\Omega}M_s(y,\vartheta)f(y)dy d\sigma(\vartheta)\\
                       &=\frac{1}{\omega_{N-1}}  \int_{\Omega} f(y) \int_{S^{N-1}} M_s(y,\vartheta)d\sigma(\vartheta)  dy = \frac{2 \kappa_{N,s}}{s} \int_{\Omega} f(y)(1-|y|^2)^{s-1}dy, 
  \end{align*}
as claimed.
\end{proof}

\subsection{Rearrangement and Symmetrization}
\label{sec:rearr-symm}

We use this section to recall some basic and fundamental properties of the Schwarz symmetrization. We refer to \cite{kesavan-2006} and \cite{zbMATH01601796} for more details. In the following, we will denote the Lebesgue measure of a subset $\Omega \subset \mathbb{R}^N$ by $|\Omega|$. \\ \\
For a measurable subset $\Omega \subset \mathbb{R}^N$ with $|\Omega|<\infty$, we denote by $\Omega^*$ the open ball with radius $r = \left( \frac{|\Omega|}{\omega_N} \right)^\frac{1}{N}$, implying that the Schwarz symmetrization is volume preserving, i.e., $|\Omega| = |\Omega^*|$. For unbounded $\Omega$, we set $\Omega^* = \mathbb{R}^N$.
\\ \\
Next, let $f: \R^N \to [0,\infty]$ be a measurable function such that its distribution satisfies
\[
\mu_f(t) = |\{ f > t \}| < \infty \quad \text{for all} \ t > 0.
\]
We define the Schwarz symmetrization $f^*: \R^N \to [0,\infty]$ of the function $f$ as
\[
f^*(x) = \sup \{ t > 0 : \mu_f(t) > \frac{\omega_{N-1}}{N} |x|^N \}.
\]
Here, as before, $\omega_{N-1}$ is the $N-1$-dimensional measure of $S^{N-1}$, so $\frac{\omega_{N-1}}{N}$ is the $N$-dimensional measure of the unit ball in $\R^N$. In particular, the Schwarz symmetrization is a generalized right inverse of the distribution function $\mu_f(t)$. One can equivalently define the Schwarz symmetrization using the Layer-Cake-Respresentation (LCR) in terms of the integral
\[
f^*(x)=\int_0^\infty\mathbf{1}_{\{|f|>t\}^*}(x)dt.
\]
We call $f$ Schwarz symmetric if $f$ coincides with $f^*$. Notice that a function $f$ is Schwarz symmetric if and only if it is radial and  non-increasing and right-continous in the radial variable.
One main property of the Schwarz symmetrization is the equimeasurablity $\mu_f(t)=\mu_{f^*}(t)$, which in particular implies that 
\[
\lVert f^* \rVert_{L^p(\R^N)} = \lVert f \rVert_{L^p(\R^N)} \qquad \text{for all $f\in L^p(\R^N)$.}
\]
In the following, let $\Omega = B_1(0)$. To define the Schwarz symmetrization $f^*:\Omega \to [0,\infty]$ of a measurable function $f: \Omega \to [0,\infty]$, we first pass to the trivial extension of $f$ on all of $\R^N$ and then use the definition given above. 

Next, we provide an $s$-harmonic mean type formula for the fractional normal derivative $\frac{u^*}{\delta^s}$ of the Schwarz symmetrization of a class of functions $u$ on $\Omega=B_1(0)$.

\begin{prop}\label{sym-explizit-boundary-expr}
  Let $s >0$, and let $u: \Omega \to [0,\infty]$ be a measurable function with the properties that
  $$
  \frac{u}{\delta^s} \in C(\overline{\Omega \setminus B_r(0)}) \quad \text{for some $r \in (0,1)$}\qquad\; \text{and}\qquad\; \essinf_{\overline \Omega}\frac{u}{\delta^s}>0.
  $$
    Then we have
  \begin{equation}
    \label{boundary-fractional-deriv-symm-formula}
  \frac{u^*}{\delta^s}= \Bigg(\frac{1}{\omega_{N-1}}\int_{S^{N-1}} \Bigl(\frac{u}{\delta^s}(\vartheta)\Bigr)^{-\frac{1}{s}} d\sigma(\vartheta)\Bigg)^{-s} \qquad  \text{on  $\partial \Omega$.}
  \end{equation}
\end{prop}

\begin{proof}
  To simplify the notation, we set $\psi: = \frac{u}{\delta^s}$, and we define the nonnegative function $h \in L^\infty(\Omega) \cap C(\Omega \setminus \{0\})$ by 
  $$
  h(x) = \psi\Big(\frac{x}{|x|}\Big)(1-|x|)^s= \psi\Big(\frac{x}{|x|}\Big)\delta^s(x) \quad \text{for $x \in \Omega \setminus \{0\}$},\qquad h(0)= \psi_0:= \max_{S^{N-1}}\psi. 
  $$
  Then $h$ is strictly decreasing in $|x| \in (0,1)$. In a first step, we show that (\ref{boundary-fractional-deriv-symm-formula}) holds for $h$ in place of $u$. To compute the Schwarz symmetrized function $h^*$, we first note that $\mu_{h}(t) = 0$ for $t\ge \psi_0$. For $t < \psi_0$, we get
\[
\{h>t\} = \Bigl \{x \in \Omega \::\: |x|< 1- \Bigl(\frac{t}{\psi(\frac{x}{|x|})}\Bigr)^{\frac{1}{s}}\Bigr\},
\]
which leads us to 
$$
  \mu_{h}(t) :=  \frac{1}{N} \int_{S^{N-1}} \Bigl(1- \Bigl(\frac{t}{\psi(\theta)}\Bigr)^{\frac{1}{s}}\Bigr)^N d \sigma(\theta).
  $$
  Therefore, for $x \in (0,1)$, the value $h^*(x)$ is implicitly given by the equation
  \begin{equation}
    \label{eq:implicit-eq}
  |x|^N = \frac{1}{\omega_{N-1}} \int_{S^{N-1}} \Bigl(1- \Bigl(\frac{h^*(x)}{\psi(\theta)}\Bigr)^{\frac{1}{s}}\Bigr)^N d \sigma(\theta).
  \end{equation}
Taylor expanding the integrand on the RHS gives
\begin{align*}
  &\frac{1}{\omega_{N-1}} \int_{S^{N-1}} \Bigl(1- \Bigl(\frac{h^*(x)}{\psi(\theta)}\Bigr)^{\frac{1}{s}}\Bigr)^N d \sigma(\theta) = 1 - \frac{N}{\omega_{N-1}}(h^*(x))^{\frac{1}{s}} \int_{S^{N-1}}\psi(\theta)^{-\frac{1}{s}}d\sigma(\theta) + o\Bigl((h^*(x))^{\frac{1}{s}}\Bigr)
\end{align*}
as $|x| \to 1$, so
$$
\Bigl(\frac{h^*(x)}{(1-|x|^N)^s}\Bigr)^{\frac{1}{s}} \frac{N}{\omega_{N-1}} \int_{S^{N-1}}\psi(\theta)^{-\frac{1}{s}}d\sigma(\theta)= 1 + o\Bigl(\Bigl(\frac{h^*(x)}{(1-|x|^N)^s}\Bigr)^{\frac{1}{s}}\Bigr)\qquad \text{as $|x| \to 1$}
$$
and therefore
\begin{equation}
  \label{eq:h-star-prop}
  \frac{h^*}{\delta^s}\equiv \lim_{|x| \to 1}\frac{h^*(x)}{(1-|x|)^s} = N \lim_{|x| \to 1}\frac{h^*(x)}{(1-|x|^N)^s}= \Bigl(\frac{1}{\omega_{N-1}}\int_{S^{N-1}}\psi(\theta)^{-\frac{1}{s}}d\sigma(\theta)\Bigr)^{-s} \qquad \text{on $\partial \Omega$},
\end{equation}
as claimed. Next, we use the notation $v_\tau$ to define the truncation of a nonnegative function $v \in L^\infty(\Omega)$ and $\tau>0$, defined by $v_\tau(x)= \min \{v(x),\tau\}$. Moreover, we note that Schwarz symmetrization commutes with truncation, i.e. we have 
\begin{equation}
  \label{eq:truncation-property}
(v_\tau)^* = (v^*)_\tau \qquad \text{for nonnegative functions $v \in L^\infty(\Omega)$ and $\tau>0$.}  
\end{equation}
Next, we fix $\eps \in (0,\frac{1}{2})$, and define the functions $Y^\pm = (1\pm \eps)h \in L^\infty(\Omega)$.
Using the above notation for truncations, we then note that the functions
  $(Y^\pm)_\tau$ are continuous on $\overline \Omega$ for $0< \tau <\tau_0:= \frac{1}{2}\min \limits_{\partial \Omega}\psi$. Moreover, it follows by construction and the definition of $\psi$ that there exists $\rho \in (0,1)$ with
  \begin{equation}
    \label{eq:boundary-comparison-est}
  (Y^-)_{\tau}(x) \le u_{\tau}(x) \le (Y^+)_{\tau}(x) \qquad \text{for $\tau \in (0,\tau_0)$ and $\rho < |x|<1$.}    
  \end{equation}
    Moreover, we may choose $\tau_1 \in (0,\tau_0)$ with the property that
  $$
  Y^+ \ge Y^- \ge \tau_1\quad \text{and}\quad u \ge \tau_1 \quad \text{on $B_{\rho}(0)$ for $\eps \in (0,1)$.}
  $$
  It thus follows from (\ref{eq:boundary-comparison-est}) that 
  $$
  (Y^-)_\tau \le u_{\tau} \le (Y^+)_{\tau} \qquad \text{on $\Omega\quad$ for $\tau \in (0,\tau_1)$.}
$$
  and therefore also, by (\ref{eq:truncation-property}) and the order preserving property of Schwarz symmetrization, 
  \begin{equation}
    \label{eq:boundary-comparison-est-all}
  \bigl((Y^-)^*\bigr)_\tau \le (u^*)_{\tau} \le \bigl( (Y^+)^*\bigr)_\tau \qquad \text{on $\Omega\quad$ for $\tau \in (0,\tau_1)$.}
  \end{equation}
  Since
  $$
  \frac{\bigl((Y^\pm)^*\bigr)_\tau}{\delta^s} \equiv \frac{(Y^\pm)^*}{\delta^s}= (1 \pm \eps)\frac{h^*}{\delta^s} \qquad \text{on $\partial \Omega$,}
  $$
  it follows from (\ref{eq:h-star-prop}) and (\ref{eq:boundary-comparison-est-all}) that on $\partial \Omega$ we have 
  $$
(1-\eps)\Bigl(\frac{1}{\omega_{N-1}}\int_{S^{N-1}}\psi(\theta)^{-\frac{1}{s}}d\sigma(\theta)\Bigr)^{-s}\le   \frac{(u^*)_\tau}{\delta^s}= \frac{u^*}{\delta^s} \le (1+\eps)\Bigl(\frac{1}{\omega_{N-1}}\int_{S^{N-1}}\psi(\theta)^{-\frac{1}{s}}d\sigma(\theta)\Bigr)^{-s}
$$
The formula (\ref{boundary-fractional-deriv-symm-formula}) now follows by letting $\eps \to 0$ in this chain of inequalities.
\end{proof}

Now let $f\in L^\infty(\Omega)$ be a nonnegative and radial function. Then $u$ and $v$ are radial functions in $\Omega$. Moreover, $v$ is nonincreasing in the radial variable, so it is Schwarz symmetric. We start with the following lemma.

\begin{cor}\label{trunc-symmetry}
  Let $s >0$, and let $u: \Omega \to [0,\infty]$ be a {\em radial} measurable function with the properties that
  $$
  \frac{u}{\delta^s} \in C(\overline{\Omega \setminus B_r(0)}) \quad \text{for some $r \in (0,1)$}\qquad\; \text{and}\qquad\; \essinf_{\overline \Omega}\frac{u}{\delta^s}>0.
  $$
  Then we have
   \begin{equation}
     \label{eq:conincide-fractional-normal-deriv}
     \frac{u}{\delta^s} = \frac{u^*}{\delta^s} \qquad \text{on $\partial \Omega$.}
     \end{equation}    
     In particular, this holds if
     $u:= u_f \in \mathscr{H}^s_0(\Omega)$ is the solution of (\ref{fractional-Dirichlet-Poisson}) for some given 
nonnegative and radial function $f \in L^\infty(\Omega)$.
   \end{cor}

\begin{proof}
The equality (\ref{eq:conincide-fractional-normal-deriv}) follows immediately from (\ref{boundary-fractional-deriv-symm-formula}) in the case where $u$ is radial. Moreover, if $u:= u_f \in \mathscr{H}^s_0(\Omega)$ is the solution of (\ref{fractional-Dirichlet-Poisson}) for some given nonnegative and radial function $f \in L^\infty(\Omega)$, then $\frac{u}{\delta^s} \in C(\overline \Omega)$ by fractional elliptic regularity (see \cite{ros2014dirichlet}), and $\inf_{\overline \Omega}\frac{u}{\delta^s}>0$ by Lemma~\ref{Hopf-lemma}.
\end{proof}

\begin{rem}
  \label{rem-boundary-equality}
While (\ref{eq:conincide-fractional-normal-deriv}) immediately follows from (\ref{boundary-fractional-deriv-symm-formula}) in the case where $u$ is radial, it can also be proved directly in this case by showing that $u = u^*$ in a neighborhood of $\partial \Omega$.  
\end{rem}

\section{The reverse  boundary Talenti inequality for radial functions}
\label{sec:reverse-bound-talent}

In this section we complete the proof of Theorem~\ref{first-main-theorem}. Throughout this section, we let $\Omega = B_1(0)$. Moreover, we let $f \in L^\infty(\Omega)$ be a given radial and nonnegative function, and as before we let $u_f \in \mathscr{H}^s_0(\Omega)$ and $u_{f^*} \in \mathscr{H}^s_0(\Omega)$ be the corresponding solutions of (\ref{fractional-Dirichlet-Poisson}) and (\ref{fractional-Dirichlet-Poisson-symmetrized}). By Lemma~\ref{sec:further-results-1-0}, we have
$$
    \int_{\partial \Omega}\Bigl(\frac{u_f}{\delta^s}-\frac{u_{f^*}}{\delta^s}\Bigr)d \sigma = \frac{2 \kappa_{N,s} \omega_{N-1}}{s}\int_{\Omega}(1-|y|^2)^{s-1}(f(y)-f^*(y))dy.
$$
Setting 
\[
r \mapsto \mathbf{k}_{N,s}(r):=\frac{2 \kappa_{N,s} \omega_{N-1}}{s}(1-r^2)^{s-1} \qquad \text{for $r \in (0,1)$,}
\]
we may use the Layer-Cake-Representations of the functions $f$ and $f^*$ to see that
  \begin{align}
    \int_{\partial \Omega}\Bigl(\frac{u_f}{\delta^s}-\frac{u_{f^*}}{\delta^s}\Bigr)d \sigma \nonumber &= \int_{\Omega}\mathbf{k}_{N,s}(|y|) (f(y)-f^*(y))dy\\
&=\int_0^\infty\int_{B_1(0)} \mathbf{k}_{N,s}(|y|)(1_{\{f>t\}}(y)-1_{\{f^*>t\}}(y))dydt \nonumber\\
&= \int_0^\infty\Bigl(\int_{\Omega_t\setminus \Omega_t^*} \mathbf{k}_{N,s}(|y|)dy-\int_{\Omega_t^*\setminus \Omega_t} \mathbf{k}_{N,s}(|y|)dy\Bigr)dt. \label{proof-first-main-thm-difference}
  \end{align}
  Here we have put $\Omega_t = \{f> t\}$, so that $\Omega_t^* = \{f^*>t\}$, and we have $\Omega_t^*= B_{r(t)}(0)$ for some $r(t) \in (0,1]$ if $|\Omega_t|= |\Omega_t^*|>0$. We also note that  
  $$
  |\Omega_t\setminus \Omega_t^*| = |\Omega_t^*\setminus \Omega_t| \qquad \text{for $t>0$.}
  $$
  We then distinguish two cases. If $|\Omega_t\setminus \Omega_t^*|= 0$ for a.e. $t \in (0,\infty)$, then
  $$
  \int_{\partial \Omega}\frac{u_{f^*}}{\delta^s}d \sigma =  \int_{\partial \Omega}\frac{u_f}{\delta^s}d \sigma =
  \int_{\partial \Omega}\frac{u^*}{\delta^s}d \sigma
  $$
  by (\ref{proof-first-main-thm-difference}) and Lemma~\ref{trunc-symmetry}. If, on the other hand,
the set $T:= \{t \in (0,\infty)\::\:   |\Omega_t\setminus \Omega_t^*|>0\}$ has positive measure, we can use the fact that the function $\mathbf{k}_{N,s}(r)$ is strictly increasing to see that  
$$
\int_{\Omega_t\setminus \Omega_t^*} \mathbf{k}_{N,s}(|y|)dy-\int_{\Omega_t^*\setminus \Omega_t} \mathbf{k}_{N,s}(|y|)dy>\bigl(|\Omega_t\setminus\Omega_t^*|-|\Omega_t^*\setminus\Omega_t|\bigr) \mathbf{k}_{N,s}(r(t))=0 \qquad \text{for $t \in T$}
$$
and therefore 
  $$
  \int_{\partial \Omega}\frac{(u_f)^*}{\delta^s}d \sigma = \int_{\partial \Omega}\frac{u_f}{\delta^s}d \sigma > \int_{\partial \Omega}\frac{u_{f^*}}{\delta^s}d \sigma 
  $$
  again by (\ref{proof-first-main-thm-difference}) and Lemma~\ref{trunc-symmetry}. So (\ref{eq:fractional-boundary-pointwise-inequality}) holds, and we have equality in (\ref{eq:fractional-boundary-pointwise-inequality}) if and only if $|\Omega_t\setminus \Omega_t^*|=0$ for a.e. $t \in (0,\infty)$, which in turn is true if and only $f$ equals $f^*$ a.e. in $\Omega$. The proof of Theorem~\ref{first-main-theorem}(i) is thus finished, and Part (ii) is a direct consequence of Part (i).

\section{The boundary Talenti inequality for a class of nonradial functions}
\label{sec:bound-talenti-ineq}

In this section, we complete the proof of Theorem~\ref{second-main-theorem}. As before, we restrict our attention to the case $\Omega = B_1(0)$. We start with the following simple observation.

\begin{lemma}
  \label{sec:further-results-1}
  For every nonnegative radial function $f \in L^\infty(\Omega) \setminus \{0\}$, the unique solution $u = u_f$ of (\ref{fractional-Dirichlet-Poisson}) satisfies
  \begin{equation}
    \label{eq:general-lowe-bound-radial}
  \frac{u_f}{\delta^s} > \frac{2 \kappa_{N,s}}{s} \|f\|_{L^1(\Omega)} \qquad \text{on $\partial \Omega$,}  
\end{equation}
where the constant $\frac{2 \kappa_{N,s}}{s}$ in this lower bound is sharp. 
\end{lemma}  

\begin{proof}
By Lemma~\ref{sec:further-results-1-0} we have 
  \begin{align*}
    \frac{u_f}{\delta^s} \equiv  \frac{2 \kappa_{N,s}}{s} \int_{\Omega} f(y)(1-|y|^2)^{s-1}dy \qquad \text{on $\partial \Omega$.}
  \end{align*}
  Since $s < 1$, it follows that $(1-|y|^2)^{s-1} > 1$ for $y \in \Omega \setminus \{0\}$, and hence (\ref{eq:general-lowe-bound-radial}) follows. The sharpness of the constant $\frac{2 \kappa_{N,s}}{s}$ follows by considering the functions
  $$
  f_\eps := \frac{1}{|B_{\eps}(0)|} 1_{B_{\eps}(0)},
  $$
  and by noting that $\|f_\eps\|_{L^1(\Omega)} = 1$ for all $\eps \in (0,1)$, while for $\vartheta \in \partial \Omega$ we have 
$$
    \frac{u_{f_\eps}}{\delta^s}(\vartheta) = \frac{1}{|B_{\eps}(0)|}  \int_{B_\eps(0)} M_s(y,\vartheta)dy \to M_s(0,\vartheta)=
    \frac{2 \kappa_{N,s}}{s} \quad \text{as $\eps \to 0$.}
 $$
\end{proof}

Next, we consider the family of auxiliary functions
\begin{equation}
  \label{eq:def-T-tau}
T_{N,\tau}: \Omega \to \R, \qquad T_{N,\tau}(\xi) = \frac{1}{\omega_{N-1}}\int_{S^{N-1}}|\vartheta-\xi|^\tau d\sigma(\vartheta).
\end{equation}

We note that $T_{N,\tau}$ is radial function which can be written as a hypergeometric function in the argument $|\xi|$ for $N \ge 2$, see e.g. \cite[Section 2]{ferone2021symmetrization}, but we don't need this fact in the following. We only need the following simple estimate.

\begin{lemma}
  \label{sec:bound-talenti-ineq-1}
For $\tau \ge 1$ and $\xi \in \Omega \setminus \{0\}$, we have $T_{N,\tau}(\xi) > T_{N,\tau}(0)=1$.  
\end{lemma}

\begin{proof}
Let $\xi \in \Omega \setminus \{0\}$. For $\vartheta \in \partial \Omega$, we then have 
\begin{align}
  1 = |\vartheta|^\tau = \bigl|\frac{1}{2}(\vartheta-\xi) + \frac{1}{2}(\vartheta+\xi)\bigr|^\tau &\le
\bigl( \frac{1}{2}|\vartheta-\xi| + \frac{1}{2}|\vartheta+\xi|\bigr)^\tau \label{strict-inequality-convexity}\\
  &\le      \frac{1}{2}|\vartheta-\xi|^\tau + \frac{1}{2}|\vartheta+ \xi|^\tau \nonumber
\end{align}
  by the convexity of the function $|\cdot|^\tau: \R^N \to \R$, which holds for $\tau \ge 1$. Moreover, the triangle inequality in (\ref{strict-inequality-convexity}) is strict for $\vartheta \in \partial \Omega \setminus \{\pm \frac{\xi}{|\xi|}\}$. Consequently,
  $$
  T_{N,\tau}(0)= 1 < \frac{1}{\omega_{N-1}}\int_{S^{N-1}}\Bigl(\frac{1}{2}|\vartheta-\xi|^\tau + \frac{1}{2}|\vartheta+ \xi|^\tau\Bigr)d \sigma(\vartheta) = \frac{1}{\omega_{N-1}}\int_{S^{N-1}}|\vartheta-\xi|^\tau d\sigma(\vartheta)= T_{N,\tau}(\xi),
  $$
  where we used a change of variables in the last step. The claim thus follows.
\end{proof}  

\begin{prop}
\label{sym-explizit-boundary-expr-green}
Let $s \in (0,1)$ and $\xi \in \Omega$. Then the Schwarz symmetrization of the Martin kernel $M_s(\xi,\cdot )$ satisfies 
\begin{align*}
M_s(\xi,\cdot )^* \equiv \Bigg(\frac{1}{\omega_{N-1}}\int_{S^{N-1}}\Big(M_s(\xi,\vartheta)\Big)^{-\frac{1}{s}}d\sigma(\vartheta)\Bigg)^{-s} = \frac{2\kappa_{N,s}}{s}(1-|\xi|^2)^{s}\Bigl(T_{N,\frac{N}{s}}(\xi) \Bigr)^{-s}    \qquad \text{on $\partial \Omega$}
\end{align*}
with the function $T_{N,\frac{N}{s}}$ defined in (\ref{eq:def-T-tau}). Moreover, we have
\begin{align}
  \label{sym-explizit-boundary-expr-green-formula-ineq}
  M_s(\xi,\cdot)^* <  \frac{2\kappa_{N,s}}{s}(1-|\xi|^2)^{s} = M_s(0,\cdot) \qquad \text{on $\partial \Omega$ for $\xi \in \Omega \setminus \{0\}$.}
\end{align}
\end{prop}

\begin{proof}
Let $\xi \in \Omega$. By (\ref{explicit-martin-ball}) and Proposition~\ref{sym-explizit-boundary-expr} applied to $u = G_s(\cdot,\xi)$, we have, on $\partial \Omega$,    
\begin{align*}
  M_s(\xi,\cdot)^* &\equiv \Bigg(\frac{1}{\omega_{N-1}}\int_{S^{N-1}} \Bigl(\frac{G_s(\vartheta,\xi)}{\delta^s(\vartheta)}\Bigr)^{-\frac{1}{s}}d\sigma(\vartheta)\Bigg)^{-s}\\
&= \Bigg(\frac{1}{\omega_{N-1}}\int_{S^{N-1}}   \Big(M_s(\xi,\vartheta)\Big)^{-\frac{1}{s}}d\sigma(\vartheta)\Bigg)^{-s}\\
                                    &= \frac{2\kappa_{N,s}}{s}(1-|\xi|^2)^{s}\Bigg(\frac{1}{\omega_{N-1}}\int_{S^{N-1}}|\xi-\vartheta|^{\frac{N}{s}} d\sigma(\vartheta)\Bigg)^{-s} = \frac{2\kappa_{N,s}}{s}(1-|\xi|^2)^{s}\Bigl(T_{N,\frac{N}{s}}(\xi) \Bigr)^{-s} 
\end{align*}
with the function $T_{N,\frac{N}{s}}$ defined in (\ref{eq:def-T-tau}), as claimed. Moreover, (\ref{sym-explizit-boundary-expr-green-formula-ineq}) follows directly from Lemma~\ref{sec:bound-talenti-ineq-1}.
\end{proof}

\begin{lemma}
\label{support-computation}  
Let $\xi \in \Omega \setminus \{0\}$ and $0 < \rho <\min \{|\xi|, 1-|\xi|\}$. Then for every nonnegative function $f \in L^\infty(\Omega)$ supported in $B_{\rho}(\xi)$ we have
  \begin{equation}
    \label{eq:upper-bound-nonradial}
  \frac{(u_f)^*}{\delta^s} \le \frac{2 \kappa_{N,s}}{s} \frac{\bigl(1-(|\xi|-\rho)^2\bigr)^s}{\bigl(1- \frac{\rho}{1-|\xi|}\bigr)^N} \|f\|_{L^1(\Omega)} \qquad \text{on $\partial \Omega$,}  
\end{equation}
\end{lemma}

\begin{proof}
To simplify the notation, we put again $u:= u_f$. On $\partial \Omega$, we have, by Lemma~\ref{sym-explizit-boundary-expr}, 
  \begin{align}
    \frac{u^*}{\delta^s} &\equiv \Bigg(\frac{1}{\omega_{N-1}}\int_{S^{N-1}} \Bigl(\frac{u}{\delta^s}(\vartheta)\Bigr)^{-\frac{1}{s}} d\sigma(\vartheta)\Bigg)^{-s}\nonumber \\
&= \Bigg(\frac{1}{\omega_{N-1}}\int_{S^{N-1}}\Bigl( \int_{B_{\rho(\xi)}}M_s(y,\vartheta)f(y)dy\Bigr)^{-\frac{1}{s}} d\sigma(\vartheta)\Bigg)^{-s} \nonumber \\
&= \frac{2 \kappa_{N,s}}{s}\Bigg(\frac{1}{\omega_{N-1}}\int_{S^{N-1}}\Bigl( \int_{B_{\rho(\xi)}}\frac{(1-|y|^2)^s}{|y-\vartheta|^N} f(y)dy\Bigr)^{-\frac{1}{s}} d\sigma(\vartheta)\Bigg)^{-s}.\label{support-computation-1}
  \end{align}
Since $0 < \rho <\min \{|\xi|, 1-|\xi|\}$ by assumption, we note that for $y \in B_\rho(\xi)$, $\vartheta \in \partial \Omega$ we have 
$$
(1-|y|^2)^s \le \bigl(1-(|\xi|-\rho)^2\bigr)^s
$$
and also, since $|\xi - \vartheta| \ge 1 - |\xi|$, 
$$
|y-\vartheta|^N \ge \Bigl(|\xi - \vartheta|- \rho \Bigr)^N \ge \Bigl(|\xi - \vartheta|- \frac{\rho}{1-|\xi|}|\xi - \vartheta| \Bigr)^N = \Bigl(1- \frac{\rho}{1-|\xi|}\Bigr)^N |\xi-\vartheta|^N.
$$
As a consequence, 
$$  
\int_{B_{\rho(\xi)}}\frac{(1-|y|^2)^s}{|y-\vartheta|^N}f(y)dy \le \frac{\bigl(1-(|\xi|-\rho)^2\bigr)^s}{\bigl(1- \frac{\rho}{1-|\xi|}\bigr)^N} \|f\|_{L^1(\Omega)} \frac{1}{|\xi-\vartheta|^N} \qquad \text{for $\vartheta \in \partial \Omega$}
$$
and therefore 
\begin{align*}
  & \frac{1}{\omega_{N-1}} \int_{S^{N-1}} \Bigl( \int_{B_{\rho(\xi)}}\frac{(1-|y|^2)^s}{|y-\vartheta|^N} f(y)dy\Bigr)^{-\frac{1}{s}}d \sigma(\vartheta)\\
  &\ge \Bigg(\|f\|_{L^1(\Omega)} \frac{\bigl(1-(|\xi|-\rho)^2\bigr)^s}{\bigl(1- \frac{\rho}{1-|\xi|}\bigr)^N} \Bigg)^{-\frac{1}{s}}   \frac{1}{\omega_{N-1}} \int_{S^{N-1}}|\xi-\vartheta|^{\frac{N}{s}}d \sigma(\vartheta)\\
  &= \Bigg(\|f\|_{L^1(\Omega)}  \frac{\bigl(1-(|\xi|-\rho)^2\bigr)^s}{\bigl(1- \frac{\rho}{1-|\xi|}\bigr)^N}\Bigg)^{-\frac{1}{s}} T_{N,\frac{N}{s}}(\xi)
     \end{align*}
with the function $T_{N,\frac{N}{s}}$ defined in (\ref{eq:def-T-tau}). Using this inequality in (\ref{support-computation-1}) together with Lemma~\ref{sec:bound-talenti-ineq-1}, we thus get that 
$$
    \frac{u^*}{\delta^s} \le \frac{2 \kappa_{N,s}}{s}\frac{\bigl(1-(|\xi|-\rho)^2\bigr)^s}{\bigl(1- \frac{\rho}{1-|\xi|}\bigr)^N} \|f\|_{L^1(\Omega)}\Bigl ( T_{N,\frac{N}{s}}(\xi) \Bigr)^{-s} \le \frac{2 \kappa_{N,s}}{s}\frac{\bigl(1-(|\xi|-\rho)^2\bigr)^s}{\bigl(1- \frac{\rho}{1-|\xi|}\bigr)^N} \|f\|_{L^1(\Omega)}
$$
on $\Omega$, as claimed.
\end{proof}

We may now complete the

\begin{proof}[Proof of Theorem~\ref{second-main-theorem}]
  Let $\xi \in \Omega \setminus \{0\}$, let $0 < \rho <\min \{|\xi|, 1-|\xi|\}$ satisfy (\ref{eq:rho-condition-intro}), and let $f \in L^\infty(\Omega) \setminus \{0\}$ be a nonnegative function supported in $B_{\rho}(\xi)$. Then Lemma~\ref{support-computation} implies that
  $$
  \frac{(u_f)^*}{\delta^s} \le \frac{2 \kappa_{N,s}}{s} \frac{\bigl(1-(|\xi|-\rho)^2\bigr)^s}{\bigl(1- \frac{\rho}{1-|\xi|}\bigr)^N} \|f\|_{L^1(\Omega)} \le \frac{2 \kappa_{N,s}}{s}\|f\|_{L^1(\Omega)} \qquad \text{on $\partial \Omega$,}
  $$
  while Lemma~\ref{sec:further-results-1} gives that 
  $$
  \frac{u_{f^*}}{\delta^s} > \frac{2 \kappa_{N,s}}{s}\|f^*\|_{L^1(\Omega)} =\frac{2 \kappa_{N,s}}{s}\|f\|_{L^1(\Omega)}.
  $$
  Combining theses inequalities yields the boundary Talenti inequality (\ref{eq:fractional-boundary-pointwise-talenti-inequality}) with $\le$ replaced by $<$, as claimed.
\end{proof}

\section{The Case $s>1$}\label{sec: Talenti boundary ineq s>1}
In this section, we investigate the validity of our previous results for the case $s>1$. We recall that, for $s=m+\sigma>1$ with $m\in\mathbb{N},\sigma\in(0,1)$, the operator $(-\Delta)^s$ can be defined via finite differences (see e.g. \cite{abatangelo-jarohs-saldana-2018}), namely,
\begin{equation}
    (-\Delta)^su(x) := \frac{c_{N,s}}{s}\int_{\mathbb{R}^N}\frac{\delta_{m+1}u(x,y)}{|y|^{N+2s}}dy, \qquad x\in\mathbb{R}^N
\end{equation}
where
\begin{equation}
    \delta_{m+1}u(x,y) := \sum_{-m-1}^{m+1}(-1)^k\binom{2(m+1)}{m+1-k}u(x+ky), \qquad \text{for} \ x,y\in\mathbb{R}^N
\end{equation}
is a finite difference of order $2(m+1)$ and $c_{N,s}$ is given as in (\ref{eq:pointwise-def-fractional-laplace}). In the following, we consider the problem (\ref{fractional-Dirichlet-Poisson}) with given $f \in L^\infty(\Omega)$, and we note that the definition of weak solutions $u \in \mathscr{H}^s_0(\Omega)$ via the property (\ref{eq:defining-weak-sol}) carries over to the case $s>1$ without any change. Here, again, $\mathscr{H}^s_0(\Omega)$ is defined as the set of functions $u \in H^s(\R^N)$ with $u \equiv 0$ on $\R^N \setminus \Omega$.
One of the main issues in the case $s>1$ is the general lack of maximum principles for higher-order fractional Laplacian (see e.g.
\cite{abatangelo-jarohs-saldana-2018}). Fortunately, since in the following we again restrict our attention to the case $\Omega= \Omega^* = B_1(0)$,  weak solutions to (\ref{fractional-Dirichlet-Poisson}) and (\ref{fractional-Dirichlet-Poisson-symmetrized}) are still unique and represented by (\ref{eq:green-function-formula}) with the explicit positive Green function given again by (\ref{green-func}), and this representation implies a positivity preserving property, see e.g. \cite{abatangelo2018green}. More precisely, the following facts are proved in \cite{abatangelo2018green}.

\begin{lemma}\label{Hopf-etc-s-greater-1}
  Let $\Omega= B_1(0)$ be the unit ball, let $f \in L^\infty(\Omega)$ nonnegative, and let $s>0$. Then the problem (\ref{fractional-Dirichlet-Poisson}) has a unique weak solution $u= u_f \in \mathscr{H}^s_0(\Omega)$ which is given by (\ref{green-func}) with the Green function $G_s$ defined by (\ref{eq:green-function-formula}). Moreover, we have $\frac{u}{\delta^s} \in C(\overline \Omega)$ and 
\[
\inf_{\overline \Omega} \frac{u}{\delta^s} >0 \qquad \text{if} \qquad u \not \equiv 0 \quad \text{in $\Omega$,}
\]
and on $\partial \Omega$ we have the representation
$$
\frac{u}{\delta^s}(\vartheta) = \int_{\Omega}M_s(y,\vartheta)f(y)dy 
$$
with the fractional Martin kernel $M_s$ given by (\ref{explicit-martin-ball}).
\end{lemma}

Remarkably, as shown in the following theorem, the inequalities in Theorem~\ref{first-main-theorem} are actually reversed in the case $s>1$.

\begin{satz}
  \label{first-main-theorem-s-greater-1}
  Let $s>1$ and $\Omega=B_1(0)\subset\mathbb{R}^N$. For every radial and nonnegative function $f\in L^\infty(\Omega)$, we have
    \begin{equation}
\label{eq-first-main-theorem-s-greater-1}
      \frac{(u_f)^*}{\delta^s}\leq\frac{u_{f^*}}{\delta^s} \qquad \text{on} \ \partial\Omega,
    \end{equation}
    and equality holds if and only if $f=f^*$ a.e. in $\Omega$.
  \end{satz}

  As a consequence, for every radial nonnegative function $f \in L^\infty(\Omega)$ with $f \not  \equiv f^*$, the pointwise strict Talenti inequality $(u_f)^* < u_{f^*}$ holds at least in a neighborhood of $\partial \Omega$.

\begin{proof}
As in the proof of Theorem~\ref{first-main-theorem} we have 
\begin{align*}
    \int_{\partial \Omega}\Bigl(\frac{u_f}{\delta^s}-\frac{u_{f^*}}{\delta^s}\Bigr)d \sigma \nonumber &= \int_{\Omega}\mathbf{k}_{N,s}(|y|) (f(y)-f^*(y))dy\\
                                                                                                      &=\int_0^\infty\int_{B_1(0)} \mathbf{k}_{N,s}(|y|)(1_{\{f>t\}}(y)-1_{\{f^*>t\}}(y))dydt
\end{align*}
with the function
\begin{equation}
\label{k-N-s-greater-1}  
\mathbf{k}_{N,s}(r) = \frac{2\kappa_{N,s}\, \omega_{N-1} }{s}(1-r^2)^{s-1} \qquad \text{for all $r \in [0,1]$,}
\end{equation}
which is now {\em strictly decreasing} and bounded with $0 \le \mathbf{k}_{N,s}(r) \le \mathbf{k}_{N,s}(0)= \frac{2\kappa_{N,s}\omega_{N-1}}{s}$ for $r \in [0,1]$. We can therefore follow the remainder of the proof of Theorem~\ref{first-main-theorem} with reversed inequalities to obtain the result.
\end{proof}

Next, we wish to consider nonradial source functions $f \in L^\infty(B)$ in the following, aiming at a generalization of Theorem~\ref{second-main-theorem} to the case $s>1$. Since Lemma~\ref{sec:further-results-1} and its proof fails in the case $s>1$, the proof of Theorem~\ref{second-main-theorem} does not fully extend to the case $s>1$. However, we still have the following variant with a stronger assumption on $\rho$.

\begin{satz}
  \label{higher-order-nonradial}
Let $s \in (1,N]$ and $\Omega=B_1(0)\subset\mathbb{R}^N$, let $\xi\in\Omega\setminus\{0\}$, and let $0<\rho<\min\{|\xi|, 1-|\xi|\}$ satisfy
    \begin{equation}\label{ineq-N>s>1-rho-xi}
        (1-(|\xi|-\rho)^2)^s\leq \Big(1-\frac{\rho}{1-|x|}\Big)^N(1-\rho^2)^{s-1}.
    \end{equation}
Then for every nonnegative function $f\in L^\infty(\Omega)$ supported in $B_\rho(\xi)$, the boundary Talenti inequality (\ref{eq-first-main-theorem-s-greater-1}) holds, and the inequality is strict if $f\not \equiv 0$.
\end{satz}

\begin{proof}
We first note that the proof of Lemma~\ref{support-computation} extends to the case $s \in (0,N]$ since Lemma~\ref{sec:bound-talenti-ineq-1} still applies with $\tau = \frac{N}{s} \ge 1$. This yields the estimate 
\begin{equation}
  \label{higher-order-nonradial-1-ineq}
    \frac{(u_f)^*}{\delta^s} \leq \frac{2\kappa_{N,s}}{s}\frac{(1-(|\xi|-\rho)^2)^s}{(1-\frac{\rho}{1-|\xi|})^N}\lVert f\rVert_{L^1(\Omega)}.
\end{equation}
Moreover, since Lemma~\ref{sec:further-results-1-0} and its proof is valid for arbitrary $s>0$, we can write the unique value of $\frac{u_f}{\delta^s}$ on $\partial \Omega$ as 
  \begin{equation*}
    \frac{u_{f^*}}{\delta^s} = \frac{2 \kappa_{N,s}}{s} \int_{\Omega} f^*(y)(1-|y|^2)^{s-1}dy.
  \end{equation*}
Since $f^*$ is supported in $B_\rho(0)$ and $s \ge 1$, it then follows that 
\begin{equation}
  \label{higher-order-nonradial-2-ineq}
    \frac{u_{f^*}}{\delta^s}  \geq \frac{2\kappa_{N,s}}{s}(1-\rho^2)^{s-1}\lVert f \rVert_{L^1(\Omega)},
  \end{equation}
and this inequality is strict if $f \not \equiv 0$. The claim now follows by combining \ref{ineq-N>s>1-rho-xi}, (\ref{higher-order-nonradial-1-ineq}) and (\ref{higher-order-nonradial-2-ineq}). 
\end{proof}

\begin{rem}{\rm 
  Theorems~\ref{first-main-theorem-s-greater-1} and \ref{higher-order-nonradial} give rise to the following open question: Does the boundary Talenti inequality (\ref{eq-first-main-theorem-s-greater-1}) hold in the case of {\em arbitrary nonradial} nonnegative functions $f\in L^\infty(\Omega)$ if $s \in (1,N]$?
Note also that we have no result for nonradial functions $f$ in the case $s > N$. }
\end{rem}

Finally, we wish to state an extension of the Martin kernel inequality given in Proposition~\ref{sym-explizit-boundary-expr-green} to the case $s \ge 1$.

\begin{prop}
\label{sym-explizit-boundary-expr-green-s-greater-1}
Let $s >0$ and $\xi \in \Omega$. Then the Schwarz symmetrization of the function $G_s(\cdot,\xi)$ satisfies 
\begin{align*}
M_s(\xi,\cdot )^* \equiv \Bigg(\frac{1}{\omega_{N-1}}\int_{S^{N-1}}\Big(M_s(\xi,\vartheta)\Big)^{-\frac{1}{s}}d\sigma(\vartheta)\Bigg)^{-s} = \frac{2\kappa_{N,s}}{s}(1-|\xi|^2)^{s}\Bigl(T_{N,\frac{N}{s}}(\xi) \Bigr)^{-s}    \qquad \text{on $\partial \Omega$}   
\end{align*}
with the function $T_{N,\frac{N}{s}}$ defined in (\ref{eq:def-T-tau}). Moreover, if $s \in (0,N]$, we have
\begin{align}
  \label{sym-explizit-boundary-expr-green-formula-ineq-s-greater-one}
  M_s(\xi,\cdot)^* <  \frac{2\kappa_{N,s}}{s}(1-|\xi|^2)^{s} = M_s(0,\cdot) \qquad \text{for $\xi \in \Omega \setminus \{0\}$.}
\end{align}
\end{prop}

\begin{proof}
  The proof is exactly the same as the proof of Proposition~\ref{sym-explizit-boundary-expr-green-s-greater-1}. For the inequality~(\ref{sym-explizit-boundary-expr-green-formula-ineq-s-greater-one}), we only have to note that $\frac{N}{s}\ge 1$ by assumption and therefore Lemma~\ref{sec:bound-talenti-ineq-1} still applies with $\tau = \frac{N}{s}$.   
\end{proof}

\noindent Institut für Mathematik \\
Goethe-Universität Frankfurt \\
Robert-Mayer-Str. 10, D-60629 Frankfurt am Main, Germany \\
\texttt{Email address: y.elkarrouchi@mathematik.uni-frankfurt.de}\\ \\

\noindent Institut für Mathematik \\
Goethe-Universität Frankfurt \\
Robert-Mayer-Str. 10, D-60629 Frankfurt am Main, Germany \\
\texttt{Email address: weth@math.uni-frankfurt.de}

\end{document}